\documentclass{article}%
\usepackage{amsmath}
\usepackage{amsfonts}
\usepackage{amssymb}
\usepackage{graphicx}%
\newtheorem{theorem}{Theorem}

\newtheorem{corollary}[theorem]{Corollary}

\newtheorem{definition}[theorem]{Definition}

\newtheorem{lemma}[theorem]{Lemma}

\newenvironment{proof}[1][Proof]{\noindent\textbf{#1.} }{\ \rule{0.5em}{0.5em}}
\begin{document}

\title{Hadamard type fractional time-delay semilinear differential equations: Delayed
Mittag-Leffler function approach}
\author{N. I. Mahmudov\\Department of Mathematics\\Eastern Mediterranean University \\Famagusta, 99628, T.R. North Cyprus \\Mersin 10, Turkey\\Email: nazim.mahmudov@emu.edu.tr}
\date{}
\maketitle

\begin{abstract}
We propose a delayed Mittag-Leffler type matrix function with logarithm, which
is an extension of the classical Mittag-Leffler type matrix function with
logarithm and delayed Mittag-Leffler type matrix function. With the help of
the delayed Mittag-Leffler type matrix function with logarithm, we give an
explicit form a of solutions to  nonhomogeneous Hadamard type fractional
time-delay linear differential equations. Moreover, we study existence
uniqueness and stability in Ulam-Hyers sense of the Hadamard type fractional
time-delay nonlinear equations.

\end{abstract}

\section{Introduction}

Mathematical descriptions of the models described through differential
equations with derivatives of non-integer orders have proved to be a very
useful instrument for modeling of various viscoelasticity cases, stability
theory, controllability theory, and other related fields. Time-delays are
often related with physico-chemical processes, electric networks, hydraulic
networks, heredity in population growth, the economy and other related
industries. In general, a peculiarity of the adequate mathematical models is
that the rate of change of these processes depends on past history.
Differential systems describing these models are called time-delay
differential equations. The qualitative theory of linear time-delay equations
is well investigated. Recently, the time-delay differential equations has been
considered. in \cite{diblik3}-\cite{diblik8}. In \cite{diblik6}-\cite{mah2}
authors derived the exact expressions of solutions of linear continuous and
discrete delay equations by proposing the concepts of delayed matrix
functions. On the other hand, stability concepts and relative controllability
problems of linear time-delay differential equations were investigated in
\cite{shuklin}-\cite{med1}.

The unification of differential equations with delay and differential
equations with fractional derivative is provided by differential equations,
including both delay and non-integer derivatives, so called time-delay
fractional differenial equations. In applications, this unification is useful
for creating highly adequate models of some systems with memory. One can
notice that works on this field involve Riemann-Liouville and Caputo type
fractional derivatives. Besides these derivatives, there is an other
fractional derivative, involving the logarithmic function, so called Hadamard
fractional derivative. For the literature on the related field of fractional
time-delay equations of Caputo type and Riemann--Liouville type, we refer the
researcher to \cite{wang2}-\cite{wang5}.

It is known that (\cite{kst}, page 235, \cite{wang5}) a solution of a Hadamard
fractional linear system%
\begin{align*}
\left(  ^{H}D_{1^{+}}^{\alpha}y\right)  \left(  t\right)   &  =\lambda
y\left(  t\right)  +f\left(  t\right)  ,\ \ t\in\left(  1,T\right]  ,h>0,\\
\left(  ^{H}I_{1^{+}}^{1-\alpha}y\right)  \left(  1^{+}\right)   &  =a\in
R,\lambda\in R,0<\alpha<1,
\end{align*}
has the form%
\[
y\left(  t\right)  =a\left(  \ln t\right)  ^{\alpha-1}E_{\alpha,\alpha}\left[
\lambda\left(  \ln t\right)  ^{\alpha}\right]  +\int_{1}^{t}\left(  \ln
\frac{t}{s}\right)  ^{\alpha-1}E_{\alpha,\alpha}\left[  \lambda\left(
\ln\frac{t}{s}\right)  ^{\alpha}\right]  f\left(  s\right)  \frac{ds}{s}.
\]
However, we find that there exists only one \cite{yang} work on the
representation of explicit solutions of Hadamard type fractional order delay
linear differential equations. In \cite{yang} authors studied the Hadamard
type fractional linear time-delay system%

\begin{align}
\left(  ^{H}D_{1^{+}}^{\alpha}y\right)  \left(  t\right)   &  =\mathcal{A}%
_{1}y\left(  t-h\right)  ,\ \ t\in\left(  1,T\right]  ,h>0,\nonumber\\
y\left(  t\right)   &  =\varphi\left(  t\right)  ,\ \ 1\leq t\leq
h,\label{ld1}\\
\left(  ^{H}I_{1^{+}}^{1-\alpha}y\right)  \left(  1^{+}\right)   &  =a\in
R^{n},\nonumber
\end{align}
where $\mathcal{A}_{01}$ is constant $n\times n$ square matrix.

Motivated by the above researches, we investigate a new class of Hadamard-type
fractional delay differential equations. We extend to consider an explicit
representation of solutions of a Hadamard type fractional time-delay
differential equation of the following form by introducing a new delayed M-L
type function with logarithm%
\begin{equation}
\left\{
\begin{array}
[c]{c}%
\left(  ^{H}D_{1^{+}}^{\alpha}y\right)  \left(  t\right)  =\mathcal{A}%
_{0}y\left(  t\right)  +\mathcal{A}_{1}y\left(  \frac{t}{h}\right)  +f\left(
t\right)  ,\ \ t\in\left(  1,T\right]  ,h>0,\\
y\left(  t\right)  =\varphi\left(  t\right)  ,\ \ \frac{1}{h}<t\leq1,\\
\left(  ^{H}I_{1^{+}}^{1-\alpha}y\right)  \left(  h^{+}\right)  =a\in
\mathbb{R}^{n},
\end{array}
\right.  \label{de11}%
\end{equation}
where $\left(  ^{H}D_{1^{+}}^{\alpha}y\right)  \left(  \cdot\right)  $ is the
Hadamard derivative of order $\alpha\in\left(  0,1\right)  $, $\mathcal{A}%
_{0},\mathcal{A}_{01}\in\mathbb{R}^{n\times n}$ denote constant matricies, and
$\varphi:\left[  \frac{1}{h},1\right]  \rightarrow\mathbb{R}^{n}$ is an
arbitrary Hadamard differentiable vector function, $f\in C\left(  \left[
1,T\right]  ,\mathbb{R}^{n}\right)  $, $T=h^{l}$ for a fixed natural number
$l$.

The second purpose of this paper is to study the existence and stability of
solutions for a Hadamard type fractional delay differential equation%
\begin{equation}
\left\{
\begin{array}
[c]{c}%
\left(  ^{H}D_{1^{+}}^{\alpha}y\right)  \left(  t\right)  =\mathcal{A}%
_{0}y\left(  t\right)  +\mathcal{A}_{1}y\left(  \frac{t}{h}\right)  +f\left(
t,y\left(  t\right)  \right)  ,\ \ t\in\left(  1,T\right]  ,h>0,\\
y\left(  t\right)  =\varphi\left(  t\right)  ,\ \ \frac{1}{h}<t\leq1,\\
\left(  ^{H}I_{1^{+}}^{1-\alpha}y\right)  \left(  1^{+}\right)  =a\in
\mathbb{R}^{n},
\end{array}
\right.  \label{de2}%
\end{equation}

At the end of this section, we state the main contribution of the paper as follows:

(i) We propose delayed M-L type functions $Y_{h,\alpha,\beta}^{\mathcal{A}%
_{0},\mathcal{A}_{1}}\left(  t,s\right)  $ with logarithms, by means of the
matrix equations (\ref{exp1}). We show that for $\mathcal{A}_{1}=\Theta$ the
function $Y_{h,\alpha,\beta}^{\mathcal{A}_{0},\mathcal{A}_{1}}\left(
t,s\right)  $ coincide with M-L type function with two parameters $\left(  \ln
t-\ln s\right)  ^{\beta-1}E_{\alpha,\beta}\left(  \mathcal{A}_{0}\left(  \ln
t-\ln s\right)  ^{\alpha}\right)  $. For $\mathcal{A}_{0}=\Theta$ delayed M-L
type function $Y_{h,\alpha,\beta}^{\mathcal{A}_{0},\mathcal{A}_{1}}\left(
t,s\right)  $ coincide with delayed M-L type matrix function with two
parameters $E_{h,\alpha,\beta}^{\mathcal{A}_{1}}\left(  \ln t-\ln h\right)
,.$introduced in (\ref{ml2}).

(ii) We explicitly write the solution of Hadamard type fractional delay linear
system (\ref{de11}) via delayed perturbation of M-L type function with
logarithm. Using this representation we study existence uniqueness and
Ulam-Hyers stability of nonlinear equation (\ref{de2}).

\section{Preliminaries}

Let $0<a<b<\infty$ and $C\left[  a,b\right]  $ be the Banach space of all
continuous functions $y:[a,b]\rightarrow R^{n}$ with the norm $\left\Vert
y\right\Vert _{C}:=\max\left\{  \left\Vert y\left(  t\right)  \right\Vert
:t\in\left[  a,b\right]  \right\}  $. For $0\leq\gamma<1$, we denote the space
$C_{\gamma,\ln}\left(  a,b\right]  $ by the weighted Banach space of the
continuous function $y:[a,b]\rightarrow R$, which is given by%
\[
C_{\gamma,\ln}\left(  a,b\right]  :=\left\{  y\left(  t\right)  :\left(
\ln\frac{t}{a}\right)  ^{\gamma}y\left(  t\right)  \in C\left[  a,b\right]
\right\}  ,
\]
endowed with the norm $\left\Vert y\right\Vert _{\gamma}:=\sup\left\{  \left(
\ln\frac{t}{a}\right)  ^{\gamma}\left\Vert y\left(  t\right)  \right\Vert
:t\in\left(  a,b\right]  \right\}  $.

The following definitions and lemmas will be used in this paper.

\begin{definition}
Hadamard fractional integral of order $\alpha\in R^{+}$ of function $y\left(
t\right)  $ is defined by
\[
\left(  ^{H}I_{a^{+}}^{\alpha}y\right)  \left(  t\right)  =\frac{1}%
{\Gamma\left(  \alpha\right)  }\int_{a}^{t}\left(  \ln\frac{t}{s}\right)
^{\alpha-1}y\left(  s\right)  \frac{ds}{s},\ 0<a<t\leq b,
\]
where $\Gamma$ is the Gamma function.
\end{definition}

\begin{definition}
Hadamard fractional derivative of order $\alpha\in\left[  n-1,n\right)  $,
$n\in Z^{+}$ of function $y\left(  t\right)  $ is defined by
\[
\left(  ^{H}D_{a^{+}}^{\alpha}y\right)  \left(  t\right)  =\frac{1}%
{\Gamma\left(  n-\alpha\right)  }\left(  t\frac{d}{dt}\right)  ^{n}\int
_{a}^{t}\left(  \ln\frac{t}{s}\right)  ^{n-\alpha+1}y\left(  s\right)
\frac{ds}{s},\ 0<a<t\leq b.
\]

\end{definition}

\begin{lemma}
If $a,\gamma,\beta>0$ then

\begin{itemize}
\item $\left(  ^{H}I_{a^{+}}^{\gamma}\left(  \ln\frac{t}{a}\right)  ^{\beta
-1}\right)  \left(  t\right)  =\dfrac{\Gamma\left(  \beta\right)  }
{\Gamma\left(  \beta+\gamma\right)  }\left(  \ln\frac{t}{a}\right)
^{\beta+\gamma-1}.$

\item $\left(  ^{H}D_{a^{+}}^{\gamma}\left(  \ln\frac{t}{a}\right)  ^{\beta
-1}\right)  \left(  t\right)  =\dfrac{\Gamma\left(  \beta\right)  }
{\Gamma\left(  \beta-\gamma\right)  }\left(  \ln\frac{t}{a}\right)
^{\beta-\gamma-1}.$

\item For $0<\beta<1$, $\left(  ^{H}D_{a^{+}}^{\beta}\left(  \ln\frac{t}%
{a}\right)  ^{\beta-1}\right)  \left(  t\right)  =0.$
\end{itemize}
\end{lemma}

\begin{definition}
\label{def:01} M-L type matrix function with two parameters $e_{\alpha,\beta
}\left(  \mathcal{A}_{0};t\right)  :\mathbb{R}\rightarrow\mathbb{R}^{n\times
n}$ is defined by
\[
e_{\alpha,\beta}\left(  \mathcal{A}_{0};t\right)  :=t^{\beta-1}E_{\alpha
,\beta}\left(  \mathcal{A}_{0};t\right)  :=t^{\beta-1}%
{\displaystyle\sum\limits_{k=0}^{\infty}}
\frac{\mathcal{A}_{0}^{k}t^{\alpha k}}{\Gamma\left(  k\alpha+\beta\right)
},\ \ \ \alpha,\beta>0,t\in\mathbb{R}.
\]

\end{definition}

\begin{definition}
\label{def:11} Two parameters delayed M-L type matrix function $E_{h,\alpha
,\beta}^{\mathcal{A}_{01}}\left(  \ln t\right)  :\mathbb{R}^{+}\rightarrow
\mathbb{R}^{n\times n}$ with logarithm is defined by
\begin{equation}
E_{h,\alpha,\beta}^{\mathcal{A}_{1}}\left(  \ln t\right)  :=\left\{
\begin{tabular}
[c]{ll}%
$\Theta,$ & $-\infty<t\leq\frac{1}{h},$\\
$I\frac{\left(  \ln t+\ln h\right)  ^{\beta-1}}{\Gamma\left(  \beta\right)
},$ & $\frac{1}{h}<t\leq1,$\\
$I\frac{\left(  \ln t+\ln h\right)  ^{\beta-1}}{\Gamma\left(  \beta\right)
}+\mathcal{A}_{1}\frac{\left(  \ln t\right)  ^{\alpha+\beta-1}}{\Gamma\left(
\alpha+\beta\right)  }+...+\mathcal{A}_{1}^{p}\frac{\left(  \ln t-\left(
p-1\right)  \ln h\right)  ^{p\alpha+\beta-1}}{\Gamma\left(  p\alpha
+\beta\right)  },$ & $h^{p-1}<t\leq h^{p}.$%
\end{tabular}
\ \ \right.  \label{ml2}%
\end{equation}

\end{definition}

Our definition of the two parameters delayed M-L type matrix function with
logarithm differs substantially from the definition given in \cite{yang}.

In order to give a definition of delayed M-L type matrix functions with
logarithm, we introduce the following matrices $Y_{\alpha,\beta,k},$
$k=0,1,2,...$%
\begin{align}
Y_{\alpha,\beta,0}\left(  t,s\right)   &  =\left(  \ln\frac{t}{s}\right)
^{\beta-1}E_{\alpha,\beta}\left(  \mathcal{A}_{0};\ln\frac{t}{s}\right)
,\ \ \nonumber\\
Y_{\alpha,\beta,1}\left(  t,sh\right)   &  =\int_{sh}^{t}e_{\alpha,\alpha
}\left(  \mathcal{A}_{0};\ln\frac{t}{r}\right)  \mathcal{A}_{1}\frac{1}%
{\Gamma\left(  \beta\right)  }\left(  \ln\frac{r}{sh}\right)  ^{\beta-1}%
\frac{dr}{r},\nonumber\\
Y_{\alpha,\beta,k}\left(  t,sh^{k}\right)   &  =\int_{sh^{k}}^{t}%
e_{\alpha,\alpha}\left(  \mathcal{A}_{0};\ln\frac{t}{r}\right)  \mathcal{A}%
_{1}Y_{\alpha,\beta,k-1}\left(  \frac{r}{h},sh^{k-1}\right)  \frac{dr}{r}.
\label{for1}%
\end{align}

\begin{definition}
\label{def:22} Let $\mathcal{A}_{0},\mathcal{A}_{1}\in\mathbb{R}^{n\times n}$
be fixed matrices and $k\in\mathbb{N}\cup\left\{  0\right\}  $. Delayed M-L
type function $Y_{h,\alpha,\beta}^{\mathcal{A}_{0},\mathcal{A}_{1}}\left(
\cdot,\cdot\right)  :\mathbb{R\times R\rightarrow R}^{n}$ with logarithm
generated by $\mathcal{A}_{0},\mathcal{A}_{1}$\ is defined by
\begin{equation}
Y_{h,\alpha,\beta}^{\mathcal{A}_{0},\mathcal{A}_{1}}\left(  t,s\right)  :=%
{\displaystyle\sum\limits_{j=0}^{\infty}}
Y_{\alpha,\beta,j}\left(  t,sh^{j}\right)  H\left(  t-sh^{j}\right)  =\left\{
\begin{tabular}
[c]{ll}%
$\Theta,$ & $-\infty<t<s,$\\
$I,$ & $t=s,$\\
$Y_{\alpha,\beta,0}\left(  t,s\right)  +Y_{\alpha,\beta,1}\left(  t,sh\right)
+...+Y_{\alpha,\beta,k}\left(  t,sh^{k}\right)  ,$ & $sh^{k}<t\leq sh^{k+1},$%
\end{tabular}
\ \ \ \ \right.  \label{exp1}%
\end{equation}
where $H\left(  t\right)  $ is a Heaviside function: $H\left(  t\right)
=\left\{
\begin{array}
[c]{c}%
1,\ \ t>0,\\
0,\ \ \ t\leq0.
\end{array}
\right.  $.
\end{definition}

\begin{lemma}
\label{lem:2}Let $a,b>-1$. For $sh^{k}<t\leq sh^{k+1},$ $k\in\mathbb{N}%
\cup\{0\}$, one has
\begin{align}
\int_{r}^{t}\left(  \ln\frac{t}{s}\right)  ^{a}\left(  \ln\frac{s}{r}\right)
^{b}\frac{ds}{s}  &  =\left(  \ln\frac{t}{r}\right)  ^{a+b+1}\mathcal{A}%
_{1}[a+1,b+1],\label{inq1}\\
\frac{1}{\Gamma\left(  1-\alpha\right)  }\int_{sh^{k}}^{t}\left(  \ln\frac
{t}{r}\right)  ^{-\alpha}Y_{\alpha,\beta,k}\left(  r,sh^{k}\right)  \frac
{dr}{r}  &  =\int_{sh^{k}}^{t}E_{\alpha,1}\left(  \mathcal{A}_{0};\ln\frac
{t}{r}\right)  \mathcal{A}_{1}Y_{\alpha,\beta,k-1}\left(  \frac{r}{h}%
,sh^{k-1}\right)  \frac{dr}{r}. \label{inq2}%
\end{align}

\end{lemma}

\begin{proof}
Let $\ln\dfrac{s}{r}=\tau\ln\dfrac{t}{r}.$ Then $s=r\left(  \dfrac{t}%
{r}\right)  ^{\tau}$, $ds=r\left(  \dfrac{t}{r}\right)  ^{\tau}\ln\dfrac{t}%
{r}d\tau$. So we have
\[
\ln\frac{t}{s}=\ln\left(  \dfrac{t}{r}\right)  ^{1-\tau}=\left(
1-\tau\right)  \ln\left(  \dfrac{t}{r}\right)  ,
\]
and
\begin{align*}
\int_{r}^{t}\left(  \ln\frac{t}{s}\right)  ^{a}\left(  \ln\frac{s}{r}\right)
^{b}\frac{ds}{s}  &  =\int_{0}^{1}\left(  \ln\frac{t}{r}\right)  ^{a}\left(
1-\tau\right)  ^{a}\left(  \ln\frac{t}{r}\right)  ^{b}\tau^{b}\ln\dfrac{t}%
{r}d\tau\\
&  =\left(  \ln\frac{t}{r}\right)  ^{a+b+1}\mathcal{B}\left[  a+1,b+1\right]
.
\end{align*}
To prove (\ref{inq2}), firstly using (\ref{inq1}) we calculate it for $k=0:$
\begin{align}
\frac{1}{\Gamma\left(  1-\alpha\right)  }\int_{s}^{t}\left(  \ln\frac{t}%
{r}\right)  ^{-\alpha}Y_{\alpha,\beta,0}\left(  r,s\right)  \frac{dr}{r}  &
=\frac{1}{\Gamma\left(  1-\alpha\right)  }\frac{1}{\Gamma\left(  \beta\right)
}\int_{s}^{t}\left(  \ln\frac{t}{r}\right)  ^{-\alpha}\left(  \ln\frac{r}%
{s}\right)  ^{\beta-1}\frac{dr}{r}\nonumber\\
&  =\frac{\left(  \ln t-\ln s\right)  ^{-\alpha+\beta}}{\Gamma\left(
-\alpha+\beta+1\right)  }. \label{form1}%
\end{align}
Similarly, for any $k\in\mathbb{N}$we have
\begin{align*}
&  \frac{1}{\Gamma\left(  1-\alpha\right)  }\int_{sh^{k}}^{t}\left(  \ln
\frac{t}{r}\right)  ^{-\alpha}Y_{\alpha,\beta,k}\left(  r,sh^{k}\right)
\frac{dr}{r}\\
&  =\frac{1}{\Gamma\left(  1-\alpha\right)  }\int_{sh^{k}}^{t}\left(  \ln
\frac{t}{r}\right)  ^{-\alpha}\int_{sh^{k}}^{r}e_{\alpha,\alpha}\left(
\mathcal{A}_{0};\ln\frac{r}{r_{1}}\right)  \mathcal{A}_{1}Y_{\alpha,\beta
,k-1}\left(  \frac{r_{1}}{h},sh^{k-1}\right)  H\left(  r-sh^{k}\right)
\frac{dr_{1}}{r_{1}}\frac{dr}{r}\\
&  =\frac{1}{\Gamma\left(  1-\alpha\right)  }\int_{sh^{k}}^{t}\int_{r_{1}}%
^{t}\left(  \ln\frac{t}{r}\right)  ^{-\alpha}e_{\alpha,\alpha}\left(
\mathcal{A}_{0};\ln\frac{r}{r_{1}}\right)  \frac{dr}{r}\mathcal{A}%
_{1}Y_{\alpha,\beta,k-1}\left(  \frac{r_{1}}{h},sh^{k-1}\right)  H\left(
r-sh^{k}\right)  \frac{dr_{1}}{r_{1}}\\
&  =\int_{sh^{k}}^{t}E_{\alpha,1}\left(  \mathcal{A}_{0};\ln\frac{t}{r_{1}%
}\right)  \mathcal{A}_{1}Y_{\alpha,\beta,k-1}\left(  \frac{r_{1}}{h}%
,sh^{k-1}\right)  \frac{dr_{1}}{r_{1}}.
\end{align*}

\end{proof}

\begin{lemma}
\label{lem:1}If $\mathcal{A}_{0}$ and $\mathcal{A}_{1}$ are permutable
matrices, then
\begin{equation}
Y_{\alpha,\beta,k}\left(  t,sh^{k}\right)  =\mathcal{A}_{1}^{k}\sum
_{n=0}^{\infty}\left(
\begin{array}
[c]{c}%
n+k\\
k
\end{array}
\right)  \mathcal{A}_{0}^{n}\frac{\left(  \ln t-\ln sh^{k}\right)
^{n\alpha+k\alpha+\beta-1}}{\Gamma\left(  n\alpha+k\alpha+\beta\right)  }.
\label{id1}%
\end{equation}

\end{lemma}

\begin{proof}
The proof is based on the inequality (\ref{inq1}). For $k=1,$we have
\begin{align*}
Y_{\alpha,\beta,1}\left(  t,sh\right)   &  =\int_{sh}^{t}e_{\alpha,\alpha
}\left(  \mathcal{A}_{0};\ln\frac{t}{r}\right)  \mathcal{A}_{1}Y_{\alpha
,\beta,0}\left(  \frac{r}{h},s\right)  \frac{dr}{r}\\
&  =\mathcal{A}_{1}\int_{sh}^{t}\sum_{n=0}^{\infty}\mathcal{A}_{0}^{n}\left(
\ln\frac{t}{r}\right)  ^{\alpha n+\alpha-1}\frac{1}{\Gamma\left(  \alpha
n+\alpha\right)  }\sum_{k=0}^{\infty}\mathcal{A}_{0}^{k}\left(  \ln\frac
{r}{sh}\right)  ^{\alpha n+\beta-1}\frac{1}{\Gamma\left(  \alpha
k+\beta\right)  }\frac{dr}{r}\\
&  =\mathcal{A}_{1}\int_{sh}^{t}\sum_{n=0}^{\infty}\sum_{k=0}^{k}%
\mathcal{A}_{0}^{k}\left(  \ln\frac{t}{r}\right)  ^{\alpha k+\alpha-1}\frac
{1}{\Gamma\left(  \alpha k+\alpha\right)  }\mathcal{A}_{0}^{n-k}\left(
\ln\frac{r}{sh}\right)  ^{\alpha\left(  n-k\right)  +\beta-1}\frac{1}%
{\Gamma\left(  \alpha\left(  n-k\right)  +\beta\right)  }\frac{dr}{r}\\
&  =\mathcal{A}_{1}\sum_{n=0}^{\infty}\sum_{k=0}^{k}\mathcal{A}_{0}^{n}%
\frac{1}{\Gamma\left(  \alpha k+\alpha\right)  \Gamma\left(  \alpha\left(
n-k\right)  +\beta\right)  }\int_{sh}^{t}\left(  \ln\frac{t}{r}\right)
^{\alpha k+\alpha-1}\left(  \ln\frac{r}{sh}\right)  ^{\alpha\left(
n-k\right)  +\beta-1}\frac{dr}{r}\\
&  =\mathcal{A}_{1}\sum_{n=0}^{\infty}\sum_{k=0}^{k}\mathcal{A}_{0}^{n}%
\frac{\left(  \ln t-\ln sh\right)  ^{\alpha n+\alpha+\beta-1}}{\Gamma\left(
\alpha k+\alpha\right)  \Gamma\left(  \alpha\left(  n-k\right)  +\beta\right)
}\mathcal{A}_{01}\left[  \alpha k+\alpha,\alpha\left(  n-k\right)
+\beta\right]  \\
&  =\mathcal{A}_{1}\sum_{n=0}^{\infty}\left(
\begin{array}
[c]{c}%
n+1\\
1
\end{array}
\right)  \mathcal{A}_{0}^{n}\frac{\left(  \ln t-\ln sh\right)  ^{\alpha
n+\alpha+\beta-1}}{\Gamma\left(  \alpha n+\alpha+\beta\right)  }.
\end{align*}
For $k=2,$ we get
\begin{align*}
Y_{\alpha,\beta,2}\left(  t,sh^{2}\right)   &  =\int_{sh^{2}}^{t}%
e_{\alpha,\alpha}\left(  \mathcal{A}_{0};\ln\frac{t}{r}\right)  \mathcal{A}%
_{1}Y_{\alpha,\beta,1}\left(  \frac{r}{h},sh\right)  dr\\
&  =\mathcal{A}_{1}\int_{sh^{2}}^{t}\sum_{n=0}^{\infty}\sum_{k=0}%
^{n}\mathcal{A}_{0}^{k}\left(  \ln\frac{t}{r}\right)  ^{k\alpha+\alpha-1}%
\frac{1}{\Gamma\left(  k\alpha+\alpha\right)  }\mathcal{A}_{1}\left(
\begin{array}
[c]{c}%
n+1-k\\
1
\end{array}
\right)  \\
&  \ \ \ \ \times\mathcal{A}_{0}^{n-k}\left(  \ln\frac{r}{sh^{2}}\right)
^{n\alpha-k\alpha+\alpha+\beta-1}\frac{1}{\Gamma\left(  n\alpha-k\alpha
+\alpha+\beta\right)  }\\
&  =\mathcal{A}_{1}^{2}\sum_{n=0}^{\infty}\sum_{k=0}^{n}\left(
\begin{array}
[c]{c}%
n+1-k\\
1
\end{array}
\right)  \mathcal{A}_{0}^{n}\frac{1}{\Gamma\left(  k\alpha+\beta\right)
\Gamma\left(  n\alpha-k\alpha+\alpha+\beta\right)  }\int_{sh^{2}}^{t}\left(
\ln\frac{t}{r}\right)  ^{k\alpha+\alpha-1}\left(  \ln\frac{r}{sh^{2}}\right)
^{n\alpha-k\alpha+\alpha+\beta-1}dr\\
&  =\mathcal{A}_{1}^{2}\sum_{n=0}^{\infty}\left(  \sum_{k=0}^{n}\left(
\begin{array}
[c]{c}%
n+1-k\\
1
\end{array}
\right)  \right)  \mathcal{A}_{0}^{n}\left(  \ln\frac{r}{sh^{2}}\right)
^{n\alpha+2\alpha+\beta-1}\frac{1}{\Gamma\left(  n\alpha+2\alpha+\beta\right)
}\\
&  =\mathcal{A}_{1}^{2}\sum_{n=0}^{\infty}\left(
\begin{array}
[c]{c}%
n+2\\
2
\end{array}
\right)  \mathcal{A}_{0}^{n}\frac{\left(  \ln t-\ln sh^{2}\right)
^{n\alpha+2\alpha+\beta-1}}{\Gamma\left(  n\alpha+2\alpha+\beta\right)  }.
\end{align*}
Using the Mathematical Induction in a similar manner we can get (\ref{id1}).
\end{proof}

According to Lemma \ref{lem:1} in the case $\mathcal{A}_{0}\mathcal{A}%
_{1}=\mathcal{A}_{1}\mathcal{A}_{0}$ delayed M-L type function $Y_{h,\alpha
,\beta}^{\mathcal{A}_{0},\mathcal{A}_{1}}\left(  t,s\right)  $ has a simple
form:
\begin{equation}
Y_{h,\alpha,\beta}^{\mathcal{A}_{0},\mathcal{A}_{1}}\left(  t,s\right)
:=\left\{
\begin{tabular}
[c]{ll}%
$\Theta,$ & $-\infty\leq t<s,$\\
$I,$ & $t=s,$\\
$%
{\displaystyle\sum\limits_{i=0}^{\infty}}
\mathcal{A}_{0}^{i}\dfrac{\left(  \ln t-\ln s\right)  ^{i\alpha+\beta-1}%
}{\Gamma\left(  i\alpha+\beta\right)  }+%
{\displaystyle\sum\limits_{i=1}^{\infty}}
\left(
\begin{array}
[c]{c}%
i\\
1
\end{array}
\right)  \mathcal{A}_{0}^{i-1}\mathcal{A}_{1}\dfrac{\left(  \ln t-\ln
sh\right)  ^{i\alpha+\beta-1}}{\Gamma\left(  i\alpha+\beta\right)  }$ & \\
$+...+%
{\displaystyle\sum\limits_{i=p}^{\infty}}
\left(
\begin{array}
[c]{c}%
i\\
p
\end{array}
\right)  \mathcal{A}_{0}^{i-p}\mathcal{A}_{1}^{p}\dfrac{\left(  \ln t-\ln
sh^{p}\right)  ^{i\alpha+\beta-1}}{\Gamma\left(  i\alpha+\beta\right)  },$ &
$sh^{p}<t\leq sh^{p+1}.$%
\end{tabular}
\ \ \ \ \ \ \ \ \ \ \ \ \right.  \label{perf1}%
\end{equation}

Next lemma shows some special cases of the delayed M-L type function.

\begin{lemma}
\label{lem:3}Let $Y_{h,\alpha,\beta}^{\mathcal{A}_{0},\mathcal{A}_{1}}\left(
t,s\right)  $ be defined by (\ref{exp1}). Then the following holds true:

\begin{description}
\item[(i)] if $\mathcal{A}_{0}=\Theta$ then $Y_{h,\alpha,\beta}^{\mathcal{A}%
_{0},\mathcal{A}_{1}}\left(  t,1\right)  =E_{h,\alpha,\beta}^{\mathcal{A}_{1}%
}\left(  \ln\frac{t}{h}\right)  ,\ \ h^{k-1}<\frac{t}{h}\leq h^{k}$,

\item[(ii)] if $\mathcal{A}_{1}=\Theta$ then $Y_{h,\alpha,\beta}%
^{\mathcal{A}_{0},\mathcal{A}_{1}}\left(  t,s\right)  =\left(  \ln\frac{t}%
{s}\right)  ^{\beta-1}E_{\alpha,\beta}\left(  \mathcal{A}_{0}\left(  \ln
\frac{t}{s}\right)  ^{\alpha}\right)  =e_{\alpha,\beta}\left(  \mathcal{A}%
_{0};\ln\frac{t}{s}\right)  ,$

\item[(iii)] if $\alpha=\beta=1$ and $\mathcal{A}_{0}\mathcal{A}%
_{1}=\mathcal{A}_{1}\mathcal{A}_{0}$ then $Y_{h,1,1}^{\mathcal{A}%
_{0},\mathcal{A}_{1}}\left(  t,s\right)  =e^{\mathcal{A}_{0}\left(  \ln t-\ln
s\right)  }e_{h}^{\mathcal{A}_{11}\left(  \ln t-\ln h\right)  }$,
$\mathcal{A}_{11}=\mathcal{A}_{1}e^{-\mathcal{A}_{0}\ln h},$ $sh^{k}<t\leq
sh^{k+1}$.
\end{description}
\end{lemma}

\begin{proof}
(i) If $\mathcal{A}_{0}=\Theta,$ then the formula (\ref{for1})
\begin{align*}
Y_{\alpha,\beta,0}\left(  t,s\right)   &  =e_{\alpha,\beta}\left(  \Theta
,\ln\frac{t}{s}\right)  =I\frac{\left(  \ln t-\ln s\right)  ^{\beta-1}}%
{\Gamma\left(  \beta\right)  },\\
Y_{\alpha,\beta,1}\left(  t,sh\right)   &  =\int_{sh}^{t}e_{\alpha,\alpha
}\left(  \Theta,\ln\frac{t}{r}\right)  \mathcal{A}_{1}Y_{0}\left(  \frac{r}%
{h},s\right)  \frac{dr}{r}=\frac{1}{\Gamma\left(  \alpha\right)  \Gamma\left(
\beta\right)  }\mathcal{A}_{1}\int_{sh}^{t}\left(  \ln\frac{t}{r}\right)
^{\alpha-1}\left(  \ln\frac{r}{sh}\right)  ^{\beta-1}\frac{dr}{r}\\
&  =\frac{1}{\Gamma\left(  \alpha\right)  \Gamma\left(  \beta\right)
}\mathcal{A}_{1}\left(  \ln\frac{t}{sh}\right)  ^{\alpha+\beta-1}%
\mathcal{B}\left[  \alpha,\beta\right]  =\frac{1}{\Gamma\left(  \alpha
+\beta\right)  }\mathcal{A}_{1}\left(  \ln\frac{t}{sh}\right)  ^{\alpha
+\beta-1},\\
Y_{\alpha,\beta,2}\left(  t,sh^{2}\right)   &  =\int_{sh^{2}}^{t}%
e_{\alpha,\alpha}\left(  \Theta,\ln\frac{t}{r}\right)  \mathcal{A}_{1}%
Y_{1}\left(  \frac{r}{h},sh\right)  dr=\frac{1}{\Gamma\left(  \alpha\right)
\Gamma\left(  \alpha+\beta\right)  }\mathcal{A}_{1}^{2}\int_{sh^{2}}%
^{t}\left(  \ln\frac{t}{r}\right)  ^{\alpha-1}\left(  \ln\frac{r}{sh^{2}%
}\right)  ^{\alpha+\beta-1}\frac{dr}{r}\\
&  =\frac{1}{\Gamma\left(  \alpha\right)  \Gamma\left(  \alpha+\beta\right)
}\mathcal{A}_{1}^{2}\left(  \ln\frac{t}{sh^{2}}\right)  ^{2\alpha+\beta
-1}\mathcal{B}\left[  \alpha,\alpha+\beta\right]  =\frac{1}{\Gamma\left(
2\alpha+\beta\right)  }\mathcal{A}_{1}^{2}\left(  \ln\frac{t}{sh^{2}}\right)
^{2\alpha+\beta-1},\\
Y_{\alpha,\beta,k}\left(  t,sh^{k}\right)   &  =\mathcal{A}_{1}^{k}\left(
\ln\frac{t}{sh^{k}}\right)  ^{k\alpha+\beta-1}\frac{1}{\Gamma\left(
k\alpha+\beta\right)  },\ \ k\geq0.
\end{align*}
So $Y_{h,\alpha,\beta}^{\mathcal{A}_{0},\mathcal{A}_{1}}\left(  t,1\right)  $
coincides with $E_{h,\alpha,\beta}^{\mathcal{A}_{1}}\left(  \ln t-\ln
h\right)  :$
\begin{align*}
Y_{h,\alpha,\beta}^{\mathcal{A}_{0},\mathcal{A}_{1}}\left(  t,1\right)   &  =%
{\displaystyle\sum\limits_{i=0}^{k}}
\mathcal{A}_{1}^{i}\left(  \ln\frac{t}{h^{i}}\right)  ^{i\alpha+\beta-1}%
\dfrac{1}{\Gamma\left(  i\alpha+\beta\right)  }=I\frac{\left(  \ln t\right)
^{\beta-1}}{\Gamma\left(  \beta\right)  }+\mathcal{A}_{1}\frac{\left(  \ln
t-\ln h\right)  ^{\alpha+\beta-1}}{\Gamma\left(  \alpha+\beta\right)
}+...+\mathcal{A}_{1}^{k}\frac{\left(  \ln t-k\ln h\right)  ^{k\alpha+\beta
-1}}{\Gamma\left(  k\alpha+\beta\right)  }\\
&  =E_{h,\alpha,\beta}^{\mathcal{A}_{1}}\left(  \ln t-\ln h\right)
,\ \ h^{k-1}<\frac{t}{h}\leq h^{k}.
\end{align*}

(ii) Trivially, from definition of $Y_{h,\alpha,\beta}^{\mathcal{A}%
_{0},\mathcal{A}_{1}}\left(  t,s\right)  $ we have: if $\mathcal{A}_{1}%
=\Theta$, then
\[
Y_{h,\alpha,\beta}^{\mathcal{A}_{0},\mathcal{A}_{1}}\left(  t,s\right)
=\left(  \ln\frac{t}{s}\right)  ^{\beta-1}E_{\alpha,\beta}\left(
\mathcal{A}_{0}\left(  \ln\frac{t}{s}\right)  ^{\alpha}\right)  .
\]

(iii) By (\ref{perf1}) for the case $\alpha=\beta=1$ and $\mathcal{A}%
_{0}\mathcal{A}_{1}=\mathcal{A}_{1}\mathcal{A}_{0}$, we have
\begin{align*}
Y_{h,1,1}^{\mathcal{A}_{0},\mathcal{A}_{1}}\left(  t,s\right)   &  =%
{\displaystyle\sum\limits_{i=0}^{\infty}}
\mathcal{A}_{0}^{i}\dfrac{\left(  \ln t-\ln s\right)  ^{i}}{\Gamma\left(
i+1\right)  }+%
{\displaystyle\sum\limits_{i=0}^{\infty}}
\left(
\begin{array}
[c]{c}%
i+1\\
1
\end{array}
\right)  \mathcal{A}_{0}^{i}\mathcal{A}_{1}\dfrac{\left(  \ln t-\ln sh\right)
^{i+1}}{\Gamma\left(  i+1\right)  }\\
&  +...+%
{\displaystyle\sum\limits_{i=0}^{\infty}}
\left(
\begin{array}
[c]{c}%
i+k\\
k
\end{array}
\right)  \mathcal{A}_{0}^{i}\mathcal{A}_{1}^{k}\dfrac{\left(  \ln t-\ln
sh^{k}\right)  ^{i+k}}{\Gamma\left(  i+k+1\right)  }\\
&  =e^{\mathcal{A}_{0}\left(  \ln t-\ln s\right)  }+e^{\mathcal{A}_{0}\left(
\ln t-\ln sh\right)  }\mathcal{A}_{1}\left(  \ln t-\ln sh\right)
+...+e^{\mathcal{A}_{0}\left(  \ln t-\ln sh^{k}\right)  }\mathcal{A}_{1}%
^{k}\frac{1}{k!}\left(  \ln t-k\ln h\right)  ^{k}\\
&  =e^{\mathcal{A}_{0}\left(  \ln t-\ln s\right)  }e_{h}^{\mathcal{A}%
_{11}\left(  \ln t-\ln h\right)  }.
\end{align*}

\end{proof}

It turns out that $Y_{h,\alpha,\beta}^{\mathcal{A}_{0},\mathcal{A}_{1}}\left(
t,s\right)  $ is a delayed perturbation of the Cauchy matrix with logarithm of
the homogeneous equation (\ref{de11}) with $f=0.$

\begin{lemma}
\label{lem:4}$Y_{h,\alpha,\beta}^{\mathcal{A}_{0},\mathcal{A}_{1}%
}:\mathbb{R\times R}\rightarrow\mathbb{R}^{n\times n}$ is a solution of
\begin{equation}
^{H}D_{h^{+}}^{\alpha}Y_{h,\alpha,\beta}^{\mathcal{A}_{0},\mathcal{A}_{1}%
}\left(  t,s\right)  =\left(  \ln\frac{t}{s}\right)  ^{-\alpha+\beta-1}%
\frac{1}{\Gamma\left(  -\alpha+\beta\right)  }+\mathcal{A}_{0}Y_{h,\alpha
,\beta}^{\mathcal{A}_{0},\mathcal{A}_{1}}\left(  t,s\right)  +\mathcal{A}%
_{1}Y_{h,\alpha,\beta}^{\mathcal{A}_{0},\mathcal{A}_{1}}\left(  \frac{t}%
{h},s\right)  . \label{form4}%
\end{equation}

\end{lemma}

\begin{proof}
According to (\ref{form1}) we have
\begin{align}
\left(  ^{H}D_{1^{+}}^{\alpha}Y_{\alpha,\beta,0}\left(  t,s\right)  \right)
\left(  t\right)   &  =\frac{1}{\Gamma\left(  1-\alpha\right)  }\left(
t\frac{d}{dt}\right)  \int_{s}^{t}\left(  \ln\frac{t}{r}\right)  ^{-\alpha
}Y_{\alpha,\beta,0}\left(  r,s\right)  \frac{dr}{r}\nonumber\\
&  =\frac{1}{\Gamma\left(  1-\alpha\right)  }\left(  t\frac{d}{dt}\right)
\int_{s}^{t}\left(  \ln\frac{t}{r}\right)  ^{-\alpha}e_{\alpha,\beta}\left(
r,s\right)  \frac{dr}{r}\nonumber\\
&  =\left(  t\frac{d}{dt}\right)  E_{\alpha,1-\alpha+\beta}\left(
\mathcal{A}_{0};\ln\frac{t}{s}\right)  \nonumber\\
&  =\left(  \ln\frac{t}{s}\right)  ^{-\alpha+\beta-1}\frac{1}{\Gamma\left(
-\alpha+\beta\right)  }+\mathcal{A}_{0}e_{\alpha,\beta}\left(  \mathcal{A}%
_{0};\ln\frac{t}{s}\right)  \nonumber\\
&  =\left(  \ln\frac{t}{s}\right)  ^{-\alpha+\beta-1}\frac{1}{\Gamma\left(
-\alpha+\beta\right)  }+\mathcal{A}_{0}Y_{\alpha,\beta,0}\left(  t,s\right)
.\label{form2}%
\end{align}
On the other hand for any $k\in\mathbb{N}$:%
\begin{align}
\left(  ^{H}D_{1^{+}}^{\alpha}Y_{\alpha,\beta,k}\left(  t,sh^{k}\right)
\right)  \left(  t\right)   &  =\frac{1}{\Gamma\left(  1-\alpha\right)
}\left(  t\frac{d}{dt}\right)  \int_{1}^{t}\left(  \ln\frac{t}{r}\right)
^{-\alpha}Y_{\alpha,\beta,k}\left(  r,s\right)  \frac{dr}{r}\nonumber\\
&  =\left(  t\frac{d}{dt}\right)  \int_{sh^{k}}^{t}E_{\alpha,1}\left(
\mathcal{A}_{0};\ln\frac{t}{r}\right)  \mathcal{A}_{1}Y_{\alpha,\beta
,k-1}\left(  \frac{r}{h},sh^{k-1}\right)  \frac{dr}{r}\nonumber\\
&  =\int_{sh^{k}}^{t}\left(  t\frac{d}{dt}\right)  E_{\alpha,1}\left(
\mathcal{A}_{0};\ln\frac{t}{r}\right)  \mathcal{A}_{1}Y_{\alpha,\beta
,k-1}\left(  \frac{r}{h},sh^{k-1}\right)  \frac{dr}{r}\nonumber\\
&  +\mathcal{A}_{1}Y_{\alpha,\beta,k-1}\left(  \frac{t}{h},sh^{k-1}\right)
\nonumber\\
&  =\mathcal{A}_{0}Y_{\alpha,\beta,k}\left(  t,sh^{k}\right)  +\mathcal{A}%
_{1}Y_{\alpha,\beta,k-1}\left(  \frac{t}{h},sh^{k-1}\right)  .\label{form3}%
\end{align}
From (\ref{form2}) and (\ref{form3}) it follows that for $sh^{k}<t\leq
sh^{k+1}$
\begin{align*}
^{H}D_{1^{+}}^{\alpha}Y_{h,\alpha,\beta}^{\mathcal{A}_{0},\mathcal{A}_{01}%
}\left(  t,s\right)   &  =\ ^{H}D_{1^{+}}^{\alpha}Y_{\alpha,\beta,0}\left(
t,s\right)  +\ ^{H}D_{1^{+}}^{\alpha}Y_{\alpha,\beta,1}\left(  t,sh\right)
+...+\ ^{H}D_{1^{+}}^{\alpha}Y_{\alpha,\beta,k}\left(  t,sh^{k}\right)  \\
&  =\left(  \ln\frac{t}{s}\right)  ^{-\alpha+\beta-1}\frac{1}{\Gamma\left(
-\alpha+\beta\right)  }+\mathcal{A}_{0}Y_{\alpha,\beta,0}\left(  t,s\right)
+\mathcal{A}_{0}Y_{\alpha,\beta,1}\left(  t,sh\right)  +\mathcal{A}%
_{1}Y_{\alpha,\beta,0}\left(  \frac{t}{h},s\right)  \\
&  +...+\mathcal{A}_{0}Y_{\alpha,\beta,k}\left(  t,sh^{k}\right)
+\mathcal{A}_{1}Y_{\alpha,\beta,k-1}\left(  \frac{t}{h},sh^{k-1}\right)  \\
&  =\left(  \ln\frac{t}{s}\right)  ^{-\alpha+\beta-1}\frac{1}{\Gamma\left(
-\alpha+\beta\right)  }++\mathcal{A}_{0}Y_{h,\alpha,\beta}^{\mathcal{A}%
_{0},\mathcal{A}_{01}}\left(  t,s\right)  +\mathcal{A}_{1}Y_{h,\alpha,\beta
}^{\mathcal{A}_{0},\mathcal{A}_{01}}\left(  \frac{t}{h},s\right)  .
\end{align*}
The proof is completed.
\end{proof}

\begin{theorem}
\label{thm:1}The solution $y(t)$ of (\ref{de11}) with zero initial condition
has a form
\[
y\left(  t\right)  =\int_{1}^{t}Y_{h,\alpha,\alpha}^{\mathcal{A}%
_{0},\mathcal{A}_{1}}\left(  t,s\right)  f\left(  s\right)  \frac{ds}%
{s},\ \ t\geq0.
\]

\end{theorem}

\begin{proof}
Assume that any solution of nonhomogeneous system $y\left(  t\right)  $ has
the form
\begin{equation}
y\left(  t\right)  =\int_{1}^{t}Y_{h,\alpha,\alpha}^{\mathcal{A}%
_{0},\mathcal{A}_{1}}\left(  t,s\right)  h\left(  s\right)  \frac{ds}%
{s},\ \ t\geq0,\label{rp1}%
\end{equation}
where $h\left(  s\right)  ,$ $1\leq s\leq t\leq T$ is an unknown continuous
vector function and $y(1)=0$. Having Hadamard fractional differentiation on
both sides of (\ref{rp1}) , for $1<t\leq h$ we have
\begin{align*}
\left(  ^{H}D_{1^{+}}^{\alpha}y\right)  \left(  t\right)   &  =\mathcal{A}%
_{0}y\left(  t\right)  +\mathcal{A}_{1}y\left(  \frac{t}{h}\right)  +f\left(
t\right)  \\
&  =\mathcal{A}_{0}\int_{1}^{t}Y_{h,\alpha,\alpha}^{\mathcal{A}_{0}%
,\mathcal{A}_{1}}\left(  t,s\right)  h\left(  s\right)  \frac{ds}%
{s}+\mathcal{A}_{1}\int_{1}^{t/h}Y_{h,\alpha,\alpha}^{\mathcal{A}%
_{0},\mathcal{A}_{1}}\left(  \frac{t}{h},s\right)  h\left(  s\right)
\frac{ds}{s}+f\left(  t\right)  \\
&  =\mathcal{A}_{0}\int_{1}^{t}Y_{h,\alpha,\alpha}^{\mathcal{A}_{0}%
,\mathcal{A}_{1}}\left(  t,s\right)  h\left(  s\right)  \frac{ds}{s}+f\left(
t\right)  .
\end{align*}
On the other hand, according to Lemma \ref{lem:2}, we have
\begin{align*}
\left(  ^{H}D_{1^{+}}^{\alpha}y\right)  \left(  t\right)   &  =\frac{1}%
{\Gamma\left(  1-\alpha\right)  }\left(  t\frac{d}{dt}\right)  \int_{1}%
^{t}\left(  \ln\frac{t}{r}\right)  ^{-\alpha}\left(  \int_{1}^{r}%
Y_{h,\alpha,\alpha}^{\mathcal{A}_{0},\mathcal{A}_{1}}\left(  r,s\right)
h\left(  s\right)  \frac{ds}{s}\right)  \frac{dr}{r}\\
&  =\frac{1}{\Gamma\left(  1-\alpha\right)  }\left(  t\frac{d}{dt}\right)
\int_{1}^{t}\int_{s}^{t}\left(  \ln\frac{t}{r}\right)  ^{-\alpha}%
Y_{h,\alpha,\alpha}^{\mathcal{A}_{0},\mathcal{A}_{1}}\left(  r,s\right)
h\left(  s\right)  \frac{dr}{r}\frac{ds}{s}\\
&  =c\left(  t\right)  +\frac{1}{\Gamma\left(  1-\alpha\right)  }\int_{1}%
^{t}\left(  t\frac{d}{dt}\right)  \int_{s}^{t}\left(  \ln\frac{t}{r}\right)
^{-\alpha}Y_{h,\alpha,\alpha}^{\mathcal{A}_{0},\mathcal{A}_{1}}\left(
r,s\right)  h\left(  s\right)  \frac{dr}{r}\frac{ds}{s}\\
&  =h\left(  t\right)  +\mathcal{A}_{0}\int_{1}^{t}Y_{h,\alpha,\alpha
}^{\mathcal{A}_{0},\mathcal{A}}\left(  t,s\right)  h\left(  s\right)
\frac{dr}{r}.
\end{align*}
Therefore, $h(t)\equiv f(t)$. The proof is completed.
\end{proof}

\begin{theorem}
\label{thm:2}Let $p=0,1,...,l$. A solution $y\in C\left(  \left(  \left(
p-1\right)  h,ph\right]  ,\mathbb{R}^{n}\right)  $ of (\ref{de11}) with $f=0$
has a form
\[
y\left(  t\right)  =Y_{h,\alpha,\alpha}^{\mathcal{A}_{0},\mathcal{A}_{1}%
}\left(  t,\frac{1}{h}\right)  a+\int_{\frac{1}{h}}^{1}Y_{h,\alpha,\alpha
}^{\mathcal{A}_{0},\mathcal{A}_{1}}\left(  t,s\right)  \left(  \left(
^{H}D_{\frac{1}{h}^{+}}^{\alpha}\varphi\right)  \left(  s\right)
-\mathcal{A}_{0}\varphi\left(  s\right)  \right)  \frac{ds}{s}.
\]

\end{theorem}

\begin{proof}
We are looking for a solution, which depends on an unknown constant $c,$ and
an Hadamard differentiable vector function $g\left(  t\right)  ,$ of the form
\[
y\left(  t\right)  =Y_{h,\alpha,\alpha}^{\mathcal{A}_{0},\mathcal{A}_{1}%
}\left(  t,\frac{1}{h}\right)  c+\int_{\frac{1}{h}}^{1}Y_{h,\alpha,\alpha
}^{\mathcal{A}_{0},\mathcal{A}_{1}}\left(  t,s\right)  g\left(  s\right)
\frac{ds}{s},
\]
Moreover, $y\left(  t\right)  $ satisfies initial conditions
\begin{align*}
y\left(  t\right)   &  =Y_{h,\alpha,\alpha}^{\mathcal{A}_{0},\mathcal{A}_{1}%
}\left(  t,\frac{1}{h}\right)  c+\int_{\frac{1}{h}}^{1}Y_{h,\alpha,\alpha
}^{\mathcal{A}_{0},\mathcal{A}_{1}}\left(  t,s\right)  g\left(  s\right)
\frac{ds}{s}:=\varphi\left(  t\right)  ,\ \ \frac{1}{h}<t\leq1,\\
\left(  ^{H}I_{\frac{1}{h}^{+}}^{1-\alpha}y\right)  \left(  \frac{1}{h}%
^{+}\right)   &  =a.
\end{align*}

We have
\begin{align*}
a  &  =\left(  ^{H}I_{\frac{1}{h}^{+}}^{1-\alpha}y\right)  \left(  \frac{1}%
{h}^{+}\right)  =\lim_{t\rightarrow\frac{1}{h}^{+}}\left(  ^{H}I_{\frac{1}%
{h}^{+}}^{1-\alpha}y\right)  \left(  t\right) \\
&  =\lim_{t\rightarrow\frac{1}{h}^{+}}\left(  \frac{1}{\Gamma\left(
1-\alpha\right)  }\int_{\frac{1}{h}}^{t}\left(  \ln t-\ln s\right)  ^{-\alpha
}Y_{h,\alpha,\alpha}^{\mathcal{A}_{0},\mathcal{A}_{1}}\left(  t,\frac{1}%
{h}\right)  c\frac{dt}{t}\right) \\
&  =\lim_{t\rightarrow\frac{1}{h}^{+}}\left(  \frac{1}{\Gamma\left(
1-\alpha\right)  }\int_{\frac{1}{h}}^{t}\left(  \ln t-\ln s\right)  ^{-\alpha
}e_{\alpha,\alpha}\left(  \mathcal{A}_{0},\ln t\right)  c\frac{dt}{t}\right)
=c.
\end{align*}

Thus $c=a$. Since $\frac{1}{h}<t\leq1$, we obtain that
\[
Y_{h,\alpha,\alpha}^{\mathcal{A}_{0},\mathcal{A}_{1}}\left(  t,s\right)
=\left\{
\begin{tabular}
[c]{ll}%
$\left(  \ln\frac{t}{s}\right)  ^{\alpha-1}E_{\alpha,\alpha}\left(
\mathcal{A}_{0}\left(  \ln\frac{t}{s}\right)  ^{\alpha}\right)  ,$ & $\frac
{1}{h}\leq s<t\leq1,\ $\\
$\Theta,$ & $t<s\leq h.$%
\end{tabular}
\ \ \ \ \ \right.
\]
Consequently, on interval $\frac{1}{h}<t\leq1$, we can easily derive
\begin{align}
\varphi\left(  t\right)   &  =Y_{h,\alpha,\alpha}^{\mathcal{A}_{0}%
,\mathcal{A}_{1}}\left(  t,\frac{1}{h}\right)  a+\int_{\frac{1}{h}}%
^{1}Y_{h,\alpha,\alpha}^{\mathcal{A}_{0},\mathcal{A}_{1}}\left(  t,s\right)
g\left(  s\right)  \frac{ds}{s}\label{q1}\\
&  =Y_{h,\alpha,\alpha}^{\mathcal{A}_{0},\mathcal{A}_{1}}\left(  t,\frac{1}%
{h}\right)  a+\int_{\frac{1}{h}}^{t}Y_{h,\alpha,\alpha}^{\mathcal{A}%
_{0},\mathcal{A}_{1}}\left(  t,s\right)  g\left(  s\right)  \frac{ds}{s}%
+\int_{t}^{1}Y_{h,\alpha,\alpha}^{\mathcal{A}_{0},\mathcal{A}_{1}}\left(
t,s\right)  g\left(  s\right)  \frac{ds}{s}\nonumber\\
&  =\left(  \ln th\right)  ^{\alpha-1}E_{\alpha,\alpha}\left(  \mathcal{A}%
_{0}\left(  \ln th\right)  ^{\alpha}\right)  a+\int_{\frac{1}{h}}^{t}\left(
\ln\frac{t}{s}\right)  ^{\alpha-1}E_{\alpha,\alpha}\left(  \mathcal{A}%
_{0}\left(  \ln\frac{t}{s}\right)  ^{\alpha}\right)  g\left(  s\right)
\frac{ds}{s}.\nonumber
\end{align}
Having differentiated (\ref{q1}) in Hadamard sense, we obtain
\begin{align*}
\left(  ^{H}D_{\frac{1}{h}^{+}}^{\alpha}\varphi\right)  \left(  t\right)   &
=\mathcal{A}_{0}\left(  \ln th\right)  ^{\alpha-1}E_{\alpha,\alpha}\left(
\mathcal{A}_{0}\left(  \ln th\right)  ^{\alpha}\right)  a+\mathcal{A}_{0}%
\int_{\frac{1}{h}}^{t}\left(  \ln\frac{t}{s}\right)  ^{\alpha-1}%
E_{\alpha,\alpha}\left(  \mathcal{A}_{0}\left(  \ln\frac{t}{s}\right)
^{\alpha}\right)  g\left(  s\right)  \frac{ds}{s}+g\left(  t\right) \\
&  =\mathcal{A}_{0}\varphi\left(  t\right)  +g\left(  t\right)  .
\end{align*}
Therefore, $g\left(  t\right)  =\left(  ^{H}D_{1^{+}}^{\alpha}\varphi\right)
\left(  t\right)  -\mathcal{A}_{0}\varphi\left(  t\right)  $ and the desired
formula holds.
\end{proof}

Combining Theorems \ref{thm:1} and \ref{thm:2} together we get the following result.

\begin{corollary}
A solution $y\in C\left(  \left[  1,T\right]  \cap\left(  h^{p-1}%
,h^{p}\right]  ,\mathbb{R}^{n}\right)  $ of (\ref{de11}) has a form
\begin{align*}
y\left(  t\right)   &  =Y_{h,\alpha,\alpha}^{\mathcal{A}_{0},\mathcal{A}_{1}%
}\left(  t,\frac{1}{h}\right)  a+\int_{\frac{1}{h}}^{1}Y_{h,\alpha,\alpha
}^{\mathcal{A}_{0},\mathcal{A}_{1}}\left(  t,s\right)  \left[  \left(
^{H}D_{\frac{1}{h}^{+}}^{\alpha}\varphi\right)  \left(  s\right)
-\mathcal{A}_{0}\varphi\left(  s\right)  \right]  \frac{ds}{s}\\
&  +\int_{1}^{t}Y_{h,\alpha,\alpha}^{\mathcal{A}_{0},\mathcal{A}_{1}}\left(
t,s\right)  f\left(  s\right)  \frac{ds}{s}.
\end{align*}

\end{corollary}

\section{Existence Uniqueness and Stability}

In this section, we consider the following equivalent integral form of the
nonlinear Cauchy problem for fractional time-delay differential equations with
Hadamard derivative (\ref{de2}):%
\begin{align}
y\left(  t\right)   &  =Y_{h,\alpha,\alpha}^{\mathcal{A}_{0},\mathcal{A}_{1}%
}\left(  t,\frac{1}{h}\right)  a+\int_{\frac{1}{h}}^{1}Y_{h,\alpha,\alpha
}^{\mathcal{A}_{0},\mathcal{A}_{1}}\left(  t,s\right)  \left[  \left(
^{H}D_{\frac{1}{h}^{+}}^{\alpha}\varphi\right)  \left(  s\right)
-\mathcal{A}_{0}\varphi\left(  s\right)  \right]  \frac{ds}{s}\nonumber\\
&  +\int_{1}^{t}Y_{h,\alpha,\alpha}^{\mathcal{A}_{0},\mathcal{A}_{1}}\left(
t,s\right)  f\left(  s,y\left(  s\right)  \right)  \frac{ds}{s}. \label{ine1}%
\end{align}

Let us introduce the conditions under which existence and uniqueness of the
integral equation $\left(  \ref{ine1}\right)  $ will be investigated.

\begin{enumerate}
\item[(A1)] $f:\left[  1,T\right]  \times\mathbb{R}\rightarrow\mathbb{R}$ be a
function such that $f\left(  t,y\right)  \in C_{\gamma,\ln}\left[  1,T\right]
$ with $\gamma<\alpha$ for any $y\in\mathbb{R}^{n};$

\item[(A2)] There exists a positive constant $L_{f}>0$ such that
\[
\left\Vert f\left(  t,y_{1}\right)  -f\left(  t,y_{2}\right)  \right\Vert \leq
L_{f}\left\Vert y_{1}-y_{2}\right\Vert ,
\]
for each $\left(  t,y_{1}\right)  ,\left(  t,y_{2}\right)  \in\left[
1,T\right]  \times\mathbb{R}^{n}$.
\end{enumerate}

From (A1) and (A2), it follows that
\[
\left\Vert f\left(  t,y\right)  \right\Vert \leq L_{f}\left\Vert y\right\Vert
+L_{2}\ \ \ \ \text{for some }L_{2}>0.
\]

To prove existence uniqueness and stability of (\ref{ine1}) we use the
following estimation of $Y_{\alpha,\beta}^{\mathcal{A}_{0},\mathcal{A}_{01}%
}\left(  t,s\right)  .$

\begin{lemma}
\label{lem:5}We have for $sh^{p}<t\leq sh^{p+1},$ $p=0,1,...,$
\[
\left\Vert Y_{h,\alpha,\beta}^{\mathcal{A}_{0},\mathcal{A}_{1}}\left(
t,s\right)  \left(  t,s\right)  \right\Vert \leq Y_{h,\alpha,\beta
}^{\left\Vert \mathcal{A}_{0}\right\Vert ,\left\Vert \mathcal{A}%
_{1}\right\Vert }\left(  t,s\right)  \leq Y_{1,\alpha,\beta}^{\left\Vert
\mathcal{A}_{0}\right\Vert ,\left\Vert \mathcal{A}_{1}\right\Vert }\left(
t,1\right)  .
\]

\end{lemma}

\begin{proof}
Indeeed,
\begin{align*}
\left\Vert Y_{h,\alpha,\beta}^{\mathcal{A}_{0},\mathcal{A}_{1}}\left(
t,s\right)  \right\Vert  &  \leq Y_{h,\alpha,\beta}^{\left\Vert \mathcal{A}%
_{0}\right\Vert ,\left\Vert \mathcal{A}_{1}\right\Vert }\left(  t,s\right)
=\sum_{k=0}^{p}%
{\displaystyle\sum\limits_{n=0}^{\infty}}
\left(
\begin{array}
[c]{c}%
n+k\\
k
\end{array}
\right)  \left\Vert \mathcal{A}_{1}\right\Vert ^{k}\left\Vert \mathcal{A}%
_{0}\right\Vert ^{n}\frac{\left(  \ln t-\ln sh^{k}\right)  ^{n\alpha
+k\alpha+\beta-1}}{\Gamma\left(  n\alpha+k\alpha+\beta\right)  }\\
&  \leq\sum_{k=0}^{p}%
{\displaystyle\sum\limits_{n=0}^{\infty}}
\left(
\begin{array}
[c]{c}%
n+k\\
k
\end{array}
\right)  \left\Vert \mathcal{A}_{1}\right\Vert ^{k}\left\Vert \mathcal{A}%
_{0}\right\Vert ^{n}\frac{\left(  \ln t\right)  ^{n\alpha+k\alpha+\beta-1}%
}{\Gamma\left(  n\alpha+k\alpha+\beta\right)  }\\
&  =\sum_{k=0}^{p}%
{\displaystyle\sum\limits_{n=0}^{\infty}}
\left(
\begin{array}
[c]{c}%
n+k\\
k
\end{array}
\right)  \left\Vert \mathcal{A}_{1}\right\Vert ^{k}\left\Vert \mathcal{A}%
_{0}\right\Vert ^{n}\frac{\left(  \ln t\right)  ^{n\alpha+k\alpha+\beta-1}%
}{\Gamma\left(  n\alpha+k\alpha+\beta\right)  }\\
&  =Y_{1,\alpha,\beta}^{\left\Vert \mathcal{A}_{0}\right\Vert ,\left\Vert
\mathcal{A}_{1}\right\Vert }\left(  t,1\right)  .
\end{align*}

\end{proof}

Our first result on existence and uniqueness of (\ref{ine1}) is based on the
Banach contraction principle.

\begin{theorem}
\label{thm:e}Assume that (A1), (A2) hold. If
\[
L_{f}\Gamma\left(  1-\gamma\right)  \left(  \ln T\right)  ^{\gamma}%
Y_{\alpha,\alpha-\gamma+1}^{\left\Vert \mathcal{A}_{0}\right\Vert ,\left\Vert
\mathcal{A}_{1}\right\Vert }\left(  T,1\right)  <1,
\]
then the Cauchy problem (\ref{de2}) has a unique solution on $[1,T]$.
\end{theorem}

\begin{proof}
We define an operator $\Theta$ on $\mathcal{B}_{r}:=\left\{  y\in
C_{\gamma,\ln}\left[  1,T\right]  :\left\Vert y\right\Vert _{\gamma}\leq
r\right\}  $ as follows
\begin{align*}
\left(  \Theta y\right)  \left(  t\right)   &  =Y_{h,\alpha,\alpha
}^{\mathcal{A}_{0},\mathcal{A}}\left(  t,\frac{1}{h}\right)  a+\int_{\frac
{1}{h}}^{1}Y_{h,\alpha,\alpha}^{\mathcal{A}_{0},\mathcal{A}_{1}}\left(
t,s\right)  \left[  \left(  ^{H}D_{\frac{1}{h}^{+}}^{\alpha}\varphi\right)
\left(  s\right)  -\mathcal{A}_{0}\varphi\left(  s\right)  \right]  \frac
{ds}{s}\\
&  +\int_{1}^{t}Y_{h,\alpha,\alpha}^{\mathcal{A}_{0},\mathcal{A}_{1}}\left(
t,s\right)  f\left(  s,y\left(  s\right)  \right)  \frac{ds}{s},
\end{align*}
where $r\geq\dfrac{M_{2}}{1-M_{1}}$,
\begin{align*}
M_{2} &  :=\left(  \ln T\right)  ^{\gamma}Y_{1,\alpha,\alpha}^{\left\Vert
\mathcal{A}_{0}\right\Vert ,\left\Vert \mathcal{A}_{1}\right\Vert }\left(
T,1\right)  \left\Vert a\right\Vert +\Gamma\left(  1-\gamma\right)  \left(
\ln T\right)  ^{\gamma}Y_{1,\alpha,\alpha-\gamma+1}^{\left\Vert \mathcal{A}%
_{0}\right\Vert ,\left\Vert \mathcal{A}_{1}\right\Vert }\left(  T,1\right)
\left\Vert \left(  ^{H}D_{\frac{1}{h}^{+}}^{\alpha}\varphi\right)
-\mathcal{A}_{0}\varphi\right\Vert _{\gamma,\ln}\\
&  \ \ \ +L_{2}\left(  \ln T\right)  ^{\gamma+1}Y_{1,\alpha,\alpha
}^{\left\Vert \mathcal{A}_{0}\right\Vert ,\left\Vert \mathcal{A}%
_{1}\right\Vert }\left(  T,1\right)  ,\\
M_{1} &  :=L_{f}\Gamma\left(  1-\gamma\right)  \left(  \ln T\right)  ^{\gamma
}Y_{\alpha,\alpha-\gamma+1}^{\left\Vert \mathcal{A}_{0}\right\Vert ,\left\Vert
\mathcal{A}\right\Vert }\left(  T,1\right)  .
\end{align*}

It is obvious that $\Theta$ is well defined due to (A1). Therefore, the
existence of a solution of the Cauchy problem (\ref{de2}) is equivalent to
that of the operator $\Theta$ has a fixed point on $\mathcal{B}_{r}$. We will
use the Banach contraction principle to prove that $\Theta$ has a fixed point.
The proof is divided into two steps.

\textit{Step 1}. $\Theta y\in\mathcal{B}_{r}$ for any $y\in\mathcal{B}_{r}$.

Indeed, for any $y\in\mathcal{B}_{r}$ and any $\delta>0,$ by (A3)
\begin{align}
\left\Vert \left(  \ln t\right)  ^{\gamma}\left(  \Theta y\right)  \left(
t\right)  \right\Vert  &  \leq\left(  \ln t\right)  ^{\gamma}Y_{1,\alpha
,\alpha}^{\left\Vert \mathcal{A}_{0}\right\Vert ,\left\Vert \mathcal{A}%
_{1}\right\Vert }\left(  t,\frac{1}{h}\right)  \left\Vert a\right\Vert
+\left(  \ln t\right)  ^{\gamma}\int_{\frac{1}{h}}^{1}Y_{h,\alpha,\alpha
}^{\left\Vert \mathcal{A}_{0}\right\Vert ,\left\Vert \mathcal{A}%
_{1}\right\Vert }\left(  t,s\right)  \left\Vert \left(  \left(  ^{H}%
D_{\frac{1}{h}^{+}}^{\alpha}\varphi\right)  \right)  \left(  s\right)
-\mathcal{A}_{0}\varphi\left(  s\right)  \right\Vert \frac{ds}{s}\nonumber\\
&  +\left(  \ln t\right)  ^{\gamma}\int_{1}^{t}Y_{h,\alpha,\alpha}^{\left\Vert
\mathcal{A}_{0}\right\Vert ,\left\Vert \mathcal{A}_{1}\right\Vert }\left(
t,s\right)  \left\Vert f\left(  s,y\left(  s\right)  \right)  \right\Vert
\frac{ds}{s} \label{est1}%
\end{align}
Firstly, we estimate the first integral:
\begin{align}
&  \left(  \ln t\right)  ^{\gamma}\int_{\frac{1}{h}}^{1}Y_{h,\alpha,\alpha
}^{\left\Vert \mathcal{A}_{0}\right\Vert ,\left\Vert \mathcal{A}%
_{1}\right\Vert }\left(  t,s\right)  \left\Vert \left(  \left(  ^{H}%
D_{\frac{1}{h}^{+}}^{\alpha}\varphi\right)  \right)  \left(  s\right)
-\mathcal{A}_{0}\varphi\left(  s\right)  \right\Vert \frac{ds}{s}\nonumber\\
&  \leq\left(  \ln t\right)  ^{\gamma}\sum_{k=0}^{p}%
{\displaystyle\sum\limits_{n=0}^{\infty}}
\left(
\begin{array}
[c]{c}%
n+k\\
k
\end{array}
\right)  \left\Vert \mathcal{A}_{1}\right\Vert ^{k}\left\Vert \mathcal{A}%
_{0}\right\Vert ^{n}\int_{\frac{1}{h}}^{1}\frac{\left(  \ln t-\ln
sh^{k}\right)  ^{n\alpha+k\alpha+\alpha-1}}{\Gamma\left(  n\alpha
+k\alpha+\alpha\right)  }\left(  \ln s\right)  ^{-\gamma}\frac{ds}%
{s}\left\Vert \left(  ^{H}D_{\frac{1}{h}^{+}}^{\alpha}\varphi\right)
-\mathcal{A}_{0}\varphi\right\Vert _{\gamma,\ln}\nonumber\\
&  \leq\Gamma\left(  1-\gamma\right)  \left(  \ln t\right)  ^{\gamma}%
\sum_{k=0}^{p}%
{\displaystyle\sum\limits_{n=0}^{\infty}}
\left(
\begin{array}
[c]{c}%
n+k\\
k
\end{array}
\right)  \left\Vert \mathcal{A}_{1}\right\Vert ^{k}\left\Vert \mathcal{A}%
_{0}\right\Vert ^{n}\frac{\left(  \ln t\right)  ^{n\alpha+k\alpha
+\alpha-\gamma}}{\Gamma\left(  n\alpha+k\alpha+\alpha-\gamma+1\right)
}\left\Vert \left(  ^{H}D_{\frac{1}{h}^{+}}^{\alpha}\varphi\right)
-\mathcal{A}_{0}\varphi\right\Vert _{\gamma,\ln}\nonumber\\
&  =\Gamma\left(  1-\gamma\right)  \left(  \ln t\right)  ^{\gamma}%
Y_{1,\alpha,\alpha-\gamma+1}^{\left\Vert \mathcal{A}_{0}\right\Vert
,\left\Vert \mathcal{A}\right\Vert }\left(  t,1\right)  \left\Vert \left(
^{H}D_{\frac{1}{h}^{+}}^{\alpha}\varphi\right)  -\mathcal{A}_{0}%
\varphi\right\Vert _{\gamma,\ln}. \label{est2}%
\end{align}
Similarly,
\begin{align}
&  \left(  \ln t\right)  ^{\gamma}\int_{\frac{1}{h}}^{t}Y_{h,\alpha,\alpha
}^{\left\Vert \mathcal{A}_{0}\right\Vert ,\left\Vert \mathcal{A}%
_{1}\right\Vert }\left(  t,s\right)  \left\Vert f\left(  s,y\left(  s\right)
\right)  \right\Vert \frac{ds}{s}\nonumber\\
&  \leq\left(  \ln t\right)  ^{\gamma}\int_{\frac{1}{h}}^{t}Y_{h,\alpha
,\alpha}^{\left\Vert \mathcal{A}_{0}\right\Vert ,\left\Vert \mathcal{A}%
_{1}\right\Vert }\left(  t,s\right)  \left(  L_{1}\left\Vert y\left(
s\right)  \right\Vert +L_{2}\right)  \frac{ds}{s}\nonumber\\
&  \leq L_{f}\Gamma\left(  1-\gamma\right)  \left(  \ln t\right)  ^{\gamma
}Y_{h,\alpha,\alpha-\gamma+1}^{\left\Vert \mathcal{A}_{0}\right\Vert
,\left\Vert \mathcal{A}_{1}\right\Vert }\left(  t,\frac{1}{h}\right)
\left\Vert y\right\Vert _{\gamma,\ln}+L_{2}\left(  \ln t\right)  ^{\gamma
+1}Y_{1,\alpha,\alpha}^{\left\Vert \mathcal{A}_{0}\right\Vert ,\left\Vert
\mathcal{A}_{1}\right\Vert }\left(  t,\frac{1}{h}\right)  . \label{est3}%
\end{align}

Inserting (\ref{est2}) and (\ref{est3}) into (\ref{est1}) we get
\begin{align*}
&  \left\Vert \left(  \ln t\right)  ^{\gamma}\left(  \Theta y\right)  \left(
t\right)  \right\Vert \\
&  \leq\left(  \ln t\right)  ^{\gamma}Y_{1,\alpha,\alpha}^{\left\Vert
\mathcal{A}_{0}\right\Vert ,\left\Vert \mathcal{A}_{1}\right\Vert }\left(
t,1\right)  \left\Vert a\right\Vert +\Gamma\left(  1-\gamma\right)  \left(
\ln t\right)  ^{\gamma}Y_{1,\alpha,\alpha-\gamma+1}^{\left\Vert \mathcal{A}%
_{0}\right\Vert ,\left\Vert \mathcal{A}_{1}\right\Vert }\left(  t,1\right)
\left\Vert ^{H}D_{1^{+}}^{\alpha}\varphi-\mathcal{A}_{0}\varphi\right\Vert
_{\gamma,\ln}\\
&  +L_{2}\left(  \ln t\right)  ^{\gamma+1}Y_{1,\alpha,\alpha}^{\left\Vert
\mathcal{A}_{0}\right\Vert ,\left\Vert \mathcal{A}_{1}\right\Vert }\left(
t,1\right)  +L_{f}\Gamma\left(  1-\gamma\right)  \left(  \ln t\right)
^{\gamma}Y_{1,\alpha,\alpha-\gamma+1}^{\left\Vert \mathcal{A}_{0}\right\Vert
,\left\Vert \mathcal{A}_{1}\right\Vert }\left(  t,1\right)  \left\Vert
y\right\Vert _{\gamma,\ln}\\
&  \leq M_{2}+M_{1}\left\Vert y\right\Vert _{\gamma,\ln}\leq M_{2}+M_{1}r\leq
r
\end{align*}

\textit{Step 2.} Let $y,z\in C_{\gamma,\ln}\left[  1,T\right]  $. Then similar
to the estimation (\ref{est3}) we get
\begin{align*}
&  \left\Vert \left(  \ln t\right)  ^{\gamma}\left(  \left(  \Theta y\right)
\left(  t\right)  -\left(  \Theta z\right)  \left(  t\right)  \right)
\right\Vert \leq\left(  \ln t\right)  ^{\gamma}\int_{1}^{t}Y_{h,\alpha,\alpha
}^{\left\Vert \mathcal{A}_{0}\right\Vert ,\left\Vert \mathcal{A}%
_{1}\right\Vert }\left(  t,s\right)  \left\Vert f\left(  s,y\left(  s\right)
\right)  -f\left(  s,z\left(  s\right)  \right)  \right\Vert \frac{ds}{s}\\
&  \leq\left(  \ln t\right)  ^{\gamma}\int_{1}^{t}\sum_{k=0}^{p}%
{\displaystyle\sum\limits_{n=0}^{\infty}}
\left(
\begin{array}
[c]{c}%
n+k\\
k
\end{array}
\right)  \left\Vert \mathcal{A}_{1}\right\Vert ^{k}\left\Vert \mathcal{A}%
_{0}\right\Vert ^{n}\frac{\left(  \ln t-\ln sh^{k}\right)  ^{n\alpha
+k\alpha+\alpha-1}}{\Gamma\left(  n\alpha+k\alpha+\alpha\right)  }\left\Vert
f\left(  s,y\left(  s\right)  \right)  -f\left(  s,z\left(  s\right)  \right)
\right\Vert \frac{ds}{s}\\
&  \leq L_{f}\left(  \ln t\right)  ^{\gamma}\int_{1}^{t}\sum_{k=0}^{p}%
{\displaystyle\sum\limits_{n=0}^{\infty}}
\left(
\begin{array}
[c]{c}%
n+k\\
k
\end{array}
\right)  \left\Vert \mathcal{A}_{1}\right\Vert ^{k}\left\Vert \mathcal{A}%
_{0}\right\Vert ^{n}\frac{\left(  \ln t-\ln s\right)  ^{n\alpha+k\alpha
+\alpha-1}}{\Gamma\left(  n\alpha+k\alpha+\alpha\right)  }\left(  \ln
s\right)  ^{-\gamma}\left(  \ln s\right)  ^{\gamma}\left\Vert y\left(
s\right)  -z\left(  s\right)  \right\Vert \frac{ds}{s}\\
&  \leq L_{f}\left(  \ln t\right)  ^{\gamma}\sum_{k=0}^{p}%
{\displaystyle\sum\limits_{n=0}^{\infty}}
\left(
\begin{array}
[c]{c}%
n+k\\
k
\end{array}
\right)  \left\Vert \mathcal{A}_{1}\right\Vert ^{k}\left\Vert \mathcal{A}%
_{0}\right\Vert ^{n}\int_{1}^{t}\frac{\left(  \ln t-\ln s\right)
^{n\alpha+k\alpha+\alpha-1}}{\Gamma\left(  n\alpha+k\alpha+\alpha\right)
}\left(  \ln s\right)  ^{-\gamma}\frac{ds}{s}\left\Vert y-z\right\Vert
_{\gamma,\ln}\\
&  \leq L_{f}\Gamma\left(  1-\gamma\right)  \left(  \ln t\right)  ^{\gamma
}\sum_{k=0}^{p}%
{\displaystyle\sum\limits_{n=0}^{\infty}}
\left(
\begin{array}
[c]{c}%
n+k\\
k
\end{array}
\right)  \left\Vert \mathcal{A}_{1}\right\Vert ^{k}\left\Vert \mathcal{A}%
_{0}\right\Vert ^{n}\frac{\left(  \ln t\right)  ^{n\alpha+k\alpha
+\alpha-\gamma}}{\Gamma\left(  n\alpha+k\alpha+\alpha-\gamma+1\right)
}\left\Vert y-z\right\Vert _{\gamma,\ln}\\
&  =L_{f}\Gamma\left(  1-\gamma\right)  \left(  \ln t\right)  ^{\gamma
}Y_{1,\alpha,\alpha-\gamma+1}^{\left\Vert \mathcal{A}_{0}\right\Vert
,\left\Vert \mathcal{A}_{1}\right\Vert }\left(  t,1\right)  \left\Vert
y-z\right\Vert _{\gamma,\ln},
\end{align*}
which implies that
\begin{equation}
\left\Vert \Theta y-\Theta z\right\Vert _{\gamma,\ln}\leq L_{f}\Gamma\left(
1-\gamma\right)  \left(  \ln T\right)  ^{\gamma}Y_{1,\alpha,\alpha-\gamma
+1}^{\left\Vert \mathcal{A}_{0}\right\Vert ,\left\Vert \mathcal{A}%
_{1}\right\Vert }\left(  T,1\right)  \left\Vert y-z\right\Vert _{\gamma,\ln
}.\label{fp}%
\end{equation}
Hence, the operator $\Theta$ is contraction on $\mathcal{B}_{r}$ and the proof
is competed by using the Banach fixed point theorem.
\end{proof}

Secondly, we discuss the Ulam-Hyers stability for the problems (\ref{de2}) by
means of integral operator given by%
\[
y\left(  t\right)  =\left(  \Theta y\right)  \left(  t\right)  ,\ \
\]
where $\Theta$ is defined by (\ref{ine1}).

Define the following nonlinear operator $Q:C_{\gamma,\ln}([1,T],\mathbb{R}%
^{n})\rightarrow C_{\gamma,\ln}([1,T],\mathbb{R}^{n})$:%
\[
Q\left(  y\right)  \left(  t\right)  :=\left(  ^{H}D_{1^{+}}^{\alpha}y\right)
\left(  t\right)  -\mathcal{A}_{0}y\left(  t\right)  -\mathcal{A}_{1}y\left(
\frac{t}{h}\right)  -f\left(  t,y\left(  t\right)  \right)  .
\]
For some $\varepsilon>0$, we look at the following inequality:
\begin{equation}
\left\Vert Q\left(  y\right)  \right\Vert _{\gamma,\ln}\leq\varepsilon.
\label{ul1}%
\end{equation}

\begin{definition}
We say that the equation (\ref{ine1}) is Ulam-Hyers stable, if there exist
$V>0$ such that for every solution $y^{\ast}\in C_{\gamma,\ln}([\frac{1}%
{h},T],\mathbb{R}^{n})$ of the inequality (\ref{ul1}), there exists a unique
solution $y\in C_{\gamma,\ln}([\frac{1}{h},T],\mathbb{R}^{n})$ of problem
(\ref{ine1}) with
\begin{equation}
\left\Vert y-y^{\ast}\right\Vert _{\gamma,\ln}\leq V\varepsilon. \label{ul2}%
\end{equation}

\end{definition}

\begin{theorem}
\label{Thm:3}Under the assumptions of Theorem \ref{thm:e}, the problem
(\ref{ine1}) is stable inUlam-Hyers sense.
\end{theorem}

\begin{proof}
Let $y\in C_{\gamma,\ln}([\frac{1}{h},T],\mathbb{R}^{n})$ be the solution of
\ the problem (\ref{ine1}). Let $y^{\ast}$ be any solution satisfying
(\ref{ul1}):
\[
\left(  ^{H}D_{1^{+}}^{\alpha}y^{\ast}\right)  \left(  t\right)
=\mathcal{A}_{0}y^{\ast}\left(  t\right)  +\mathcal{A}_{1}y^{\ast}\left(
\frac{t}{h}\right)  +f\left(  t,y^{\ast}\left(  t\right)  \right)  +Q\left(
y^{\ast}\right)  \left(  t\right)  .
\]
So
\[
y^{\ast}\left(  t\right)  =\Theta\left(  y^{\ast}\right)  \left(  t\right)
+\int_{1}^{t}Y_{h,\alpha,\alpha}^{\mathcal{A}_{0},\mathcal{A}_{1}}\left(
t,s\right)  Q\left(  y^{\ast}\right)  \left(  s\right)  \frac{ds}{s}.
\]
It follows that
\begin{align*}
&  \left(  \ln t\right)  ^{\gamma}\left\Vert \Theta\left(  y^{\ast}\right)
\left(  t\right)  -y^{\ast}\left(  t\right)  \right\Vert \leq\left(  \ln
t\right)  ^{\gamma}\int_{1}^{t}\left\Vert Y_{h,\alpha,\alpha}^{\mathcal{A}%
_{0},\mathcal{A}_{1}}\left(  t,s\right)  \right\Vert \left\Vert Q\left(
y^{\ast}\right)  \left(  s\right)  \right\Vert \frac{ds}{s}\\
&  \leq\left(  \ln T\right)  ^{\gamma+1}Y_{1,\alpha,\alpha}^{\left\Vert
\mathcal{A}_{0}\right\Vert ,\left\Vert \mathcal{A}_{1}\right\Vert }\left(
T,1\right)  \varepsilon.
\end{align*}
Therefore, we deduce by the fixed-point property (\ref{fp}) of the operator
$\Theta,$ that
\begin{align}
\left(  \ln t\right)  ^{\gamma}\left\Vert y\left(  t\right)  -y^{\ast}\left(
t\right)  \right\Vert  &  \leq\left(  \ln t\right)  ^{\gamma}\left\Vert
\Theta\left(  y\right)  \left(  t\right)  -\Theta\left(  y^{\ast}\right)
\left(  t\right)  \right\Vert +\left(  \ln t\right)  ^{\gamma}\left\Vert
\Theta\left(  y^{\ast}\right)  \left(  t\right)  -y^{\ast}\left(  t\right)
\right\Vert \nonumber\\
&  \leq L_{f}\Gamma\left(  1-\gamma\right)  \left(  \ln t\right)  ^{\gamma
}Y_{1,\alpha,\alpha-\gamma+1}^{\left\Vert \mathcal{A}_{0}\right\Vert
,\left\Vert \mathcal{A}_{1}\right\Vert }\left(  t,1\right)  \left\Vert
y-y^{\ast}\right\Vert _{\gamma,\ln}+\left(  \ln T\right)  ^{\gamma
+1}Y_{1,\alpha,\alpha}^{\left\Vert \mathcal{A}_{0}\right\Vert ,\left\Vert
\mathcal{A}_{1}\right\Vert }\left(  T,1\right)  \varepsilon, \label{ul7}%
\end{align}
and
\[
\left\Vert y-y^{\ast}\right\Vert _{\gamma,\ln}\leq\frac{\left(  \ln T\right)
^{\gamma+1}Y_{1,\alpha,\alpha}^{\left\Vert \mathcal{A}_{0}\right\Vert
,\left\Vert \mathcal{A}_{1}\right\Vert }\left(  T,1\right)  }{1-L_{f}%
\Gamma\left(  1-\gamma\right)  \left(  \ln T\right)  ^{\gamma}Y_{1,\alpha
,\alpha-\gamma+1}^{\left\Vert \mathcal{A}_{0}\right\Vert ,\left\Vert
\mathcal{A}_{1}\right\Vert }\left(  T,1\right)  }\varepsilon.
\]
Thus, the problem (\ref{de2}) is Ulam-Hyers stable with.
\[
V=\frac{\left(  \ln T\right)  ^{\gamma+1}Y_{1,\alpha,\alpha}^{\left\Vert
\mathcal{A}_{0}\right\Vert ,\left\Vert \mathcal{A}_{1}\right\Vert }\left(
T,1\right)  }{1-L_{f}\Gamma\left(  1-\gamma\right)  \left(  \ln T\right)
^{\gamma}Y_{1,\alpha,\alpha-\gamma+1}^{\left\Vert \mathcal{A}_{0}\right\Vert
,\left\Vert \mathcal{A}_{1}\right\Vert }\left(  T,1\right)  }.
\]

\end{proof}

\section{Example}

In this section, we give an example to illustrate the obtained theoretical
result. Let $\alpha=0.3$, $h=1.2,$ $k=4.$ Consider%
\begin{align}
\left(  ^{H}D_{1^{+}}^{0.3}y\right)  \left(  t\right)   &  =\mathcal{A}%
_{1}y\left(  \frac{t}{1.2}\right)  ,\ \ t\in\left(  1,2.0736\right]
,\nonumber\\
y\left(  t\right)   &  =\varphi\left(  t\right)  ,\ \ \frac{1}{1.2}\leq
t\leq1,\nonumber\\
\left(  ^{H}I_{\frac{1}{1.2}^{+}}^{0.7}y\right)  \left(  \frac{1}{1.2}%
^{+}\right)   &  =a\in R^{n}, \label{ex2}%
\end{align}
where%
\[
\mathcal{A}_{1}=\left(
\begin{tabular}
[c]{ll}%
$2$ & $1$\\
$3$ & $5$%
\end{tabular}
\ \ \ \right)  ,\ \ \ a=\left(
\begin{array}
[c]{c}%
1\\
2
\end{array}
\right)
\]
The solution of (\ref{ex2}) can be represented by $Y_{h,\alpha,\alpha
}^{\mathcal{A}_{0},\mathcal{A}_{1}}\left(  t,\frac{1}{h}\right)
=E_{h,\alpha,\alpha}^{\mathcal{A}_{1}}\left(  \ln t\right)  $%
\[
y\left(  t\right)  =E_{h,\alpha,\alpha}^{\mathcal{A}_{1}}\left(  \ln t\right)
a+\int_{\frac{1}{h}}^{1}E_{h,\alpha,\alpha}^{\mathcal{A}_{1}}\left(  \ln
\frac{t}{s}\right)  \left(  ^{H}D_{\frac{1}{h}^{+}}^{\alpha}\varphi\right)
\left(  s\right)  \frac{ds}{s},
\]
where%

\[
E_{h,\alpha,\alpha}^{\mathcal{A}_{1}}\left(  \ln t\right)  =\left\{
\begin{tabular}
[c]{ll}%
$0,$ & $-\infty<x\leq\frac{1}{1.2},$\\
$I\dfrac{\left(  \ln t+\ln1.2\right)  ^{-0.7}}{\Gamma\left(  0.3\right)  },$ &
$\frac{1}{1.2}<x\leq1,$\\
$I\dfrac{\left(  \ln t+\ln1.2\right)  ^{-0.7}}{\Gamma\left(  0.3\right)
}+\mathcal{A}_{1}\dfrac{\left(  \ln t\right)  ^{-0.4}}{\Gamma\left(
0.6\right)  },$ & $1<x\leq1.2,$\\
$I\dfrac{\left(  \ln t+\ln1.2\right)  ^{-0.7}}{\Gamma\left(  0.3\right)
}+\mathcal{A}_{1}\dfrac{\left(  \ln t\right)  ^{-0.4}}{\Gamma\left(
0.6\right)  }+\mathcal{A}_{1}^{2}\dfrac{\left(  \ln t-\ln1.2\right)  ^{-0.1}%
}{\Gamma\left(  0.9\right)  },$ & $1.2<x\leq\left(  1.2\right)  ^{2},$\\
$I\dfrac{\left(  \ln t+\ln1.2\right)  ^{-0.7}}{\Gamma\left(  0.3\right)
}+\mathcal{A}_{1}\dfrac{\left(  \ln t\right)  ^{-0.4}}{\Gamma\left(
0.6\right)  }+\mathcal{A}_{1}^{2}\dfrac{\left(  \ln t-\ln1.2\right)  ^{-0.1}%
}{\Gamma\left(  0.9\right)  }$ & \\
$+\mathcal{A}_{1}^{3}\dfrac{\left(  \ln t-2\ln1.2\right)  ^{0.2}}%
{\Gamma\left(  1.2\right)  }$ & $\left(  1.2\right)  ^{2}<x\leq\left(
1.2\right)  ^{3},$\\
$I\dfrac{\left(  \ln t+\ln1.2\right)  ^{-0.7}}{\Gamma\left(  0.3\right)
}+\mathcal{A}_{1}\dfrac{\left(  \ln t\right)  ^{-0.4}}{\Gamma\left(
0.6\right)  }+\mathcal{A}_{1}^{2}\dfrac{\left(  \ln t-\ln1.2\right)  ^{-0.1}%
}{\Gamma\left(  0.9\right)  }$ & \\
$+\mathcal{A}_{1}^{3}\dfrac{\left(  \ln t-2\ln1.2\right)  ^{0.2}}%
{\Gamma\left(  1.2\right)  }+\mathcal{A}_{1}^{4}\dfrac{\left(  \ln
t-3\ln1.2\right)  ^{0.5}}{\Gamma\left(  1.5\right)  }$ & $\left(  1.2\right)
^{3}<x\leq\left(  1.2\right)  ^{4}.$%
\end{tabular}
\ \ \ \right.
\]

\end{document}